\documentclass[a4paper,12pt]{amsart}
\usepackage[utf8]{inputenc}
\usepackage[british]{babel}
\usepackage{amsmath}
\usepackage{amssymb}
\usepackage{amsthm}
\usepackage[comma,square,numbers]{natbib}
\usepackage{geometry}
\usepackage{hyperref}

\newcommand{\mc}{\mathcal}
\newcommand{\eps}{\epsilon}

\newtheorem{theorem}{Theorem}[section]
\newtheorem{corollary}[theorem]{Corollary}
\newtheorem{lemma}[theorem]{Lemma}
\newtheorem{proposition}[theorem]{Proposition}

\theoremstyle{definition}
\newtheorem{definition}[theorem]{Definition}

\theoremstyle{remark}
\newtheorem{remark}[theorem]{Remark}

\begin{document}

\title[Repulsive Mean-Field]{Repulsive Mean-Field Couplings and Self-Consistent Transfer Operators}

\author[R. Castorrini]{Roberto Castorrini}
\address{(Roberto Castorrini) DEIM, University of Tuscia (Viterbo) \& Dipartimento di Matematica, Universit\`a di Pisa}
\email{roberto.castorrini@unitus.it}

\author[S. Galatolo]{Stefano Galatolo}
\address{(Stefano Galatolo) Dipartimento di Matematica, Universit\`a di Pisa -- Largo B.~Pontecorvo 5 -- 56127 Pisa --  Italy}
\email{stefano.galatolo@unipi.it}

\author[M. Tanzi]{Matteo Tanzi}
\address{(Matteo Tanzi) Department of Mathematics, King's College London}
\email{matteo.tanzi@kcl.ac.uk}

\date{}

\begin{abstract}
    We study a one-dimensional self-consistent transfer operator generated by a repulsive mean-field coupling. The interaction admits a natural interpretation as a myopic relocation dynamics balancing congestion and uniformly distributed resources. Using the quantile representation of probability measures, combined with self-consistent transfer operators methods, we prove the existence and uniqueness of a physically relevant invariant measure for the system and characterize it explicitly through a nonlinear fixed-point equation. We then show that the corresponding finite-particle equilibrium converges to the continuum equilibrium with optimal order in Wasserstein distance and establish a quantitative thermodynamic limit.

\end{abstract}

\maketitle

\section{Introduction}

Mean-field coupled maps combine simple microscopic dynamics with an interaction through a collective observable. Despite this elementary description, they can display synchronization, clustering, collective chaos, phase transitions, and ergodicity breaking; see, among others, \cite{Kan,NK,PK,Just,Just2} and the review \cite{BUMI}. When the number of components tends to infinity, the state of the population is described by a probability measure and its evolution is governed by a nonlinear operator acting on measures, usually called a \emph{self-consistent transfer operator} (STO) \cite{K,Bl11,BUMI}. The invariant measures of the STO describe macroscopic equilibria of the thermodynamic-limit dynamics. Their existence, uniqueness, stability, and relation with the underlying finite-particle system are therefore central questions.

Much of the rigorous theory developed so far concerns expanding or hyperbolic uncoupled maps and couplings weak enough that the STO remains close to a linear transfer operator. This includes results for tent maps, expanding circle maps, smooth expanding maps, Anosov diffeomorphisms, systems with noise, and intermittent maps \cite{K,SB,bal,ST,BLS,Gal,BK}. More recently, the strong-coupling regime has been approached through synchronization arguments \cite{ST2}, bifurcation and phase-transition techniques \cite{BL}, the spectral analysis of the differential of an STO \cite{CGT}, and nonlinear cone contractions \cite{CGTcones}. These developments make it possible to study genuinely nonlinear regimes, but the existing methods and examples are still largely motivated by uncoupled dynamics with expansion or mixing.

In this paper we investigate a complementary setting in which the microscopic map is \emph{contractive}, while the mean-field interaction is \emph{repulsive}. More precisely, the uncoupled dynamics is generated by a strictly increasing contraction $T:I\to I$, whereas the interaction is the rank-based map
\[
g_{\varepsilon,\mu}(x)=(1-\varepsilon)x+\varepsilon F^-_\mu(x),
\qquad \varepsilon\in(0,1),
\]
where $F^-_\mu(x)=\mu([0,x))$. The corresponding STO is
\[
\mathcal L_\varepsilon\mu=(T\circ g_{\varepsilon,\mu})_*\mu.
\]
The two parts of the dynamics act in opposite directions: the contraction $T$ tends to concentrate the population, whereas the coupling separates agents in crowded regions and pushes the population toward a rank--position balance. This competition produces a nontrivial equilibrium even though the uncoupled map alone is dissipative. Our results hold for every $\varepsilon\in(0,1)$, with no weak-coupling assumption; indeed, increasing the strength of the repulsive interaction improves the macroscopic contraction estimate obtained below.

The model has several natural interpretations. If resources are uniformly distributed on $I$, then $F^-_\mu(x)-x$ measures the cumulative excess of agents relative to resources to the left of $x$. The update $g_{\varepsilon,\mu}$ is the unique myopic best response balancing a quadratic relocation cost with a quadratic rank--position mismatch. It can thus serve as a stylized model for congestion avoidance, spatial competition, or the redistribution of agents competing for a homogeneous resource, in the spirit of large-population and mean-field-game models \cite{LachapelleWolfram2011,LasryLions2007,HuangMalhameCaines2006}. For an atomless state, $F_\mu$ is also the monotone optimal transport from $\mu$ to the uniform measure, and $g_{\varepsilon,\mu}$ is the associated displacement-interpolation step \cite{McCann1997,Villani2009,Santambrogio2015}. The subsequent action of $T$ may be interpreted as a dissipative adjustment or an exogenous pull toward a stable microscopic state.

\medskip
\noindent\textbf{Main results.}
The main observation is that the nonlinearity of the coupling becomes particularly simple in quantile coordinates. If $\mu$ is atomless and $Q_\mu$ denotes its quantile function, then
\[
Q_{(g_{\varepsilon,\mu})_*\mu}(u)
   =(1-\varepsilon)Q_\mu(u)+\varepsilon u.
\]
Consequently, for atomless probability measures $\mu$ and $\nu$,
\[
d_W\bigl(\mathcal L_\varepsilon\mu,
          \mathcal L_\varepsilon\nu\bigr)
\leq \lambda^{-1}(1-\varepsilon)d_W(\mu,\nu),
\]
where $\lambda^{-1}<1$ is a Lipschitz bound for $T$ and $d_W$ is the $1$-Wasserstein distance. Thus the STO is globally contracting on the physically relevant class, rather than merely locally contracting near an equilibrium. We deduce exponential convergence from every absolutely continuous initial measure to a unique absolutely continuous invariant measure $\mu^\varepsilon$. This equilibrium is characterized explicitly by the scalar fixed-point equation
\[
Q^\varepsilon(u)
=T\bigl((1-\varepsilon)Q^\varepsilon(u)+\varepsilon u\bigr).
\]

The restriction to atomless measures is essential rather than technical. Because the coupling uses the lower-rank convention $F^-_\mu$, Dirac masses are also invariant when supported at fixed points of $T$. The system therefore exhibits coexistence between an atomic equilibrium and the unique non-atomic equilibrium selected by absolutely continuous initial data. This distinction clarifies the precise sense in which the latter is the physical invariant measure.

We next connect the continuum description with the deterministic $N$-particle dynamics. The continuum quantile $Q^\varepsilon$ produces an exact, strictly ordered particle equilibrium by sampling it at the ranks $i/N$. This equilibrium is unique within the strictly ordered sector and attracts every strictly ordered configuration exponentially fast. Its empirical measure converges to $\mu^\varepsilon$ at the optimal order $N^{-1}$ in $d_W$, with explicit two-sided bounds. Finally, we prove a quantitative thermodynamic limit for arbitrary absolutely continuous initial data and compatible empirical approximations. The estimate is uniform over all discrete times, not only on a fixed finite time horizon: the one-step consistency error does not accumulate because it is damped by the same contraction that governs the continuum evolution.

\textbf{Organization of the paper.}
In Section~\ref{subsec:repulsive-motivation} we derive the repulsive interaction from a resource--congestion model and discuss its interpretations in terms of myopic best response and optimal transport. Section~\ref{sec:self-rep} introduces the STO generated by the contracting map, proves the Wasserstein contraction principle, and characterizes the physical invariant measure. Section~\ref{sec:finiteN} studies the finite-particle system, its ordered equilibrium and convergence to it, the optimal approximation of the continuum equilibrium, and the uniform-in-time thermodynamic limit.

%%%%%%%%%%%%%%%%%%Self-repelling

\section{A resource-congestion model and its repulsive coupling}
\label{subsec:repulsive-motivation}

Let us consider a continuum of (indistinguishable) agents moving on the one-dimensional
resource space $I=[0,1]$. The macroscopic configuration of the population is described by
a probability measure $\mu\in\mathcal P(I)$, where $\mu(A)$ represents the fraction of agents
whose position lies in a Borel set $A\subset I$.
 For $\varepsilon \in (0,1]$, and $h(x,z)=\mathbf{1}_{[0,x)}(z)$, we will consider the following mean field coupling for the dynamics of these agents \begin{equation*} g _{\varepsilon ,\mu }(x)=(1-\varepsilon )x+\varepsilon\int_{I} h(x,z) d\mu (z). \end{equation*}
In this section we  show that this interaction can be interpreted, in a suitable sense, as a repulsive coupling. Such a mechanism emerges when agents adjust their positions so as to maximize their share of a resource that is homogeneously distributed over space, while facing congestion effects due to competition with nearby agents. Examples include animals spreading out while grazing, beachgoers seeking personal space, or firms locating to avoid direct competition for the same customer base.

Throughout this section, we denote by
\[
F_\mu^-(x):=\mu([0,x)),
\qquad x\in I,
\]
the left-continuous cumulative distribution function of \(\mu\). Thus
\[
\int_I h(x,z)\,d\mu(z)=F_\mu^-(x).
\]

Assume that resources  are homogeneously distributed
over $I$, i.e.\ proportional to Lebesgue measure $m$ on $I$. Then the fraction of total resources
available in the prefix interval $[0,x)$ is exactly $m([0,x))=x$, whereas the fraction of agents
in the same prefix is $F_\mu^-(x)$. This suggests measuring a  {cumulative} congestion (misalignment)
at position $x$ by
\begin{equation}
\Delta_\mu(x):=F_\mu^-(x)-x
=\mu([0,x))-m([0,x)).
\label{eq:excess-mass}
\end{equation}
The sign of $\Delta_\mu(x)$ has the following interpretation:
\begin{itemize}
\item if $\Delta_\mu(x)>0$, then the left portion $[0,x)$ contains a {larger} fraction of agents
than its share of resources (overcrowding on the left);
\item if $\Delta_\mu(x)<0$, then $[0,x)$ is {under}-occupied relative to resources (scarcity of agents on the left).
\end{itemize}
In either case, an agent located near $x$ experiences an incentive to relocate so as to reduce
the imbalance between  {population mass} and  {resource mass}. This is consistent with a broad
class of ``aversion/congestion'' modelling principles in large-population games and crowd dynamics,
where agents incur costs increasing with crowding and/or seek personal space
(see, e.g., \cite{LachapelleWolfram2011} and the mean field game framework of \cite{LasryLions2007,HuangMalhameCaines2006}).

\noindent{\bf A myopic relocation game with adjustment costs.}
We now show how the specific coupling
\begin{equation}
g_{\varepsilon,\mu}(x)\;=\;(1-\varepsilon)x+\varepsilon F_\mu^-(x),
\qquad \varepsilon\in(0,1],
\label{eq:coupling}
\end{equation}
arises as a  {one-step best response} under a simple (but explicit) microscopic model.

Fix a macroscopic state $\mu$ and consider an infinitesimal agent currently at position $x\in I$.
In the mean-field regime each individual agent has negligible mass, neglects its own impact on $\mu$, and treats $F_\mu^-$
as exogenous data. We interpret $r:=F_\mu^-(x)\in[0,1]$ as the agent's  {rank} (its population percentile)
at the current state. If $\mu$ has atoms, agents at the same position share the same lower rank $F_\mu^-(x)$; negligible individual mass does not exclude such atoms. Under homogeneous resources, a natural ``balanced'' configuration is one in which
 {rank equals position}: an agent at percentile $r$ should occupy location $r$, so that the share of agents
to the left matches the share of resources to the left. Hence the rank-position discrepancy
$r-y$ at a candidate new location $y$ quantifies instantaneous misalignment with respect to the
resource distribution.

We therefore postulate the following myopic relocation cost:
\begin{equation}
J_{\varepsilon,\mu}(y\,;\,x)
\;:=\;\frac{1-\varepsilon}{2}\,|y-x|^2
\;+\;\frac{\varepsilon}{2}\,|y-F_\mu^-(x)|^2,
\qquad y\in I,
\label{eq:myopic-cost}
\end{equation}
which trades off a quadratic  {adjustment cost} (relocation friction) against a quadratic
 {congestion/misalignment cost} (deviation from the balanced location associated to the current rank).
The parameter $\varepsilon$ controls the relative weight of these effects: small $\varepsilon$ corresponds
to high inertia (or large moving costs), while $\varepsilon\approx 1$ corresponds to near-instantaneous
adjustment; $\varepsilon=1$ gives full adjustment.

\medskip
\noindent\textbf{Best response.}
Since \eqref{eq:myopic-cost} is strictly convex in $y$, it admits a unique minimizer on $I$. A direct
computation gives the unconstrained minimizer as the solution of
\[
(1-\varepsilon)(y-x)+\varepsilon(y-F_\mu^-(x))=0,
\]
hence
\begin{equation}
y^\star\;=\;(1-\varepsilon)x+\varepsilon F_\mu^-(x)\;=\;g_{\varepsilon,\mu}(x),
\label{eq:best-response}
\end{equation}
which is exactly \eqref{eq:coupling}. Since it is a convex combination of $x,F_\mu^-(x)\in I$, it belongs to $I$ and is also the constrained minimizer. In particular,
\begin{equation}
g_{\varepsilon,\mu}(x)-x=\varepsilon\big(F_\mu^-(x)-x\big)=\varepsilon\,\Delta_\mu(x),
\label{eq:drift}
\end{equation}
so the direction of motion is determined by the sign of the cumulative imbalance \eqref{eq:excess-mass}:
if the left side is overcrowded ($\Delta_\mu(x)>0$), the agent moves right, and if it is under-occupied
($\Delta_\mu(x)<0$), the agent moves left.

%\paragraph{Mean-field evolution and equilibrium.}
If all agents simultaneously apply the best response \eqref{eq:best-response}, then the macroscopic state updates by
push-forward:
\begin{equation}
\mu^+\;=\;(g_{\varepsilon,\mu})_*\mu,
\label{eq:pushforward}
\end{equation}
yielding a deterministic mean-field evolution.

\noindent{ \bf A repulsive coupling.}
To connect \eqref{eq:coupling} to a literal repulsion mechanism, suppose that $\mu$ is absolutely continuous,
$d\mu(x)=\rho(x)\,dx$ with density $\rho\ge 0$.
Then $(F_\mu^-)'(x)=\rho(x)$ for a.e.\ $x$, and differentiating \eqref{eq:coupling} yields
\begin{equation}
g_{\varepsilon,\mu}'(x)\;=\;(1-\varepsilon)+\varepsilon\,\rho(x)\qquad\text{for a.e.\ }x\in I.
\label{eq:derivative}
\end{equation}
Hence, in regions where $\rho(x)>1$ (density higher than the uniform resource density), one has
$g_{\varepsilon,\mu}'(x)>1$, meaning that small separations between nearby agents are  {expanded}
by the update: local configurations are pushed apart. Conversely, where $\rho(x)<1$, distances are locally
contracted, increasing the density along transported trajectories where $0<\rho(x)<1$. In this sense, the coupling generates a local repulsive effect
in crowded regions together with a compensating drift toward
under-populated ones, consistent with competition for local resources.

The map $F_\mu^-(x)$ is nondecreasing and, in one dimension, coincides with the monotone rearrangement
transporting $\mu$ to the uniform distribution; in fact it is the optimal transport map for quadratic cost under standard
regularity assumptions \cite{Villani2009,Santambrogio2015}. The convex combination
\[
g_{\varepsilon,\mu}=(1-\varepsilon)\mathrm{Id}+\varepsilon F_\mu^-
\]
therefore corresponds to a  {displacement interpolation} step along the Wasserstein geodesic connecting
$\mu$ to the uniform distribution \cite{McCann1997}. Since \(F_\mu^-\) transports \(\mu\) onto the uniform measure,
the map \(g_{\varepsilon,\mu}\) is precisely the interpolation between the identity transport and the optimal transport toward
the resource distribution. This provides an additional variational interpretation:
agents move in the direction of homogenizing their density with respect to the resources (matching resources) while incurring quadratic moving costs,
exactly as encoded in \eqref{eq:myopic-cost}.

The mechanism above is intentionally minimalist: rather than solving a full dynamic programming problem as in
standard mean field games \cite{LasryLions2007,HuangMalhameCaines2006}, we postulate a one-step (myopic) relocation
with an explicit congestion/aversion signal. Nevertheless, the economic content is analogous to the crowd aversion
and congestion effects emphasized in mean-field-game-based crowd models, notably in
\cite{LachapelleWolfram2011}, where rational agents anticipate and react to density-dependent costs.
Here the interaction is encoded via the CDF $F_\mu^-$, producing a tractable repulsive coupling that
drives an atomless population toward a balanced (uniform) spatial configuration under the coupling-only evolution \eqref{eq:pushforward}.

\section{Self-consistent dynamics generated by a contracting map}
\label{sec:self-rep}

In this section we study the statistical properties of the self-consistent transfer operator associated with the repulsive coupling introduced in Section~\ref{subsec:repulsive-motivation}. Our main goal is to establish convergence to a physically relevant invariant measure under suitable assumptions on the underlying dynamics. The proof combines contraction estimates in the Wasserstein metric with a one-dimensional quantile representation, which will also provide an explicit characterization of the physical equilibrium and its approximation by finite particle systems.

\subsection{Setting}

Let \(I=[0,1]\), and denote by \(\mathcal P(I)\) the space of Borel
probability measures on \(I\). We denote by \(m\) Lebesgue measure on \(I\).

Let \(T\in C^1(I,I)\) be a strictly increasing embedding such that, for
some \(\lambda>1\),
\[
0<T'(x)\le \lambda^{-1}<1,
\qquad x\in I.
\]
In particular, \(T\) is a \(C^1\) diffeomorphism from \(I\) onto its image
\(T(I)\subset I\). For $\epsilon \in (0,1)$ and for
\(\mu\in\mathcal P(I)\), define
\[
F_\mu^-(x):=\mu([0,x)).
\]
The repulsive coupling map is
\[
g_{\epsilon,\mu}(x)=(1-\epsilon)x+\epsilon F_\mu^-(x).
\]

The self-consistent transfer operator (STO) associated to $T$ and $g_{\epsilon,\mu}$ 
is the nonlinear operator
\[
\mc L : \mathcal P(I) \to \mathcal P(I),
\qquad
\mc L(\mu)
:=
L_T\big(L_{g_{\epsilon,\mu}}(\mu)\big)
=
(T \circ g_{\epsilon,\mu})_* \mu .
\]

For a measurable map \(f:I\to I\), we write
\[
L_f\mu:=f_*\mu.
\]
Equivalently, for every Borel set \(A\subset I\),
\[
(L_f\mu)(A)=\mu(f^{-1}(A)).
\]
If \(\mu=\psi\,dx\) with \(\psi\in C^0(I)\), then
\[
F_\mu^-(x)=\int_0^x\psi(z)\,dz,
\]
so \(g_{\epsilon,\mu}\in C^1(I,I)\) and
\[
g'_{\epsilon,\mu}(x)=(1-\epsilon)+\epsilon\psi(x)\ge 1-\epsilon>0.
\]
Moreover \(g_{\epsilon,\mu}(0)=0\) and \(g_{\epsilon,\mu}(1)=1\), hence
\(g_{\epsilon,\mu}\) is a \(C^1\) diffeomorphism of \(I\). For such continuous densities,
$\mc L$ acts in the following way.

Let $F:=T\circ g_{\epsilon,\mu}$. 
Since $F$ is a $C^1$ diffeomorphism from $I$ onto $T(I)$, the transfer operator 
$\mc L = L_F$ acts on densities as
\[
\mc L \psi (x)
=
\frac{\psi\!\left(F^{-1}(x)\right)}
{\left|F'\!\left(F^{-1}(x)\right)\right|}
\]
for $x\in T(I)$.
Using the chain rule,
\[
F'(y)
=
T'\!\left(g_{\epsilon,\mu}(y)\right)
\, g'_{\epsilon,\mu}(y),
\]
and observing that
\[
F^{-1}(x)
=
g_{\epsilon,\mu}^{-1}\!\left(T^{-1}(x)\right),
\]
we obtain
\[
\mc L \psi(x)
=
\frac{
\psi\!\left(g_{\epsilon,\mu}^{-1}(T^{-1}(x))\right)
}{
T'\!\left(T^{-1}(x)\right)
\, g'_{\epsilon,\mu}\!\left(g_{\epsilon,\mu}^{-1}(T^{-1}(x))\right)
}.
\]
Since 
\[
g'_{\epsilon,\mu}(y)
=
(1-\epsilon)+\epsilon \psi(y),
\]
this becomes
\[
\mc L \psi(x)
=
\frac{
\psi\!\left(g_{\epsilon,\mu}^{-1}(T^{-1}(x))\right)
}{
T'\!\left(T^{-1}(x)\right)
\left[
(1-\epsilon)
+
\epsilon \psi\!\left(g_{\epsilon,\mu}^{-1}(T^{-1}(x))\right)
\right]
}.
\]
On the other hand, \((\mathcal L\psi)(x)=0\) for \(x\notin T(I)\).

%%%%%%%%%%%%%%%%%%%%%%%%

%%%%%%%%%%Wasserstain dist%%%%%%%%%

\subsection{The Wasserstein distance}
\label{subsec:wasserstein}
In the following sections we prove convergence to equilibrium by studying the action of $\mc L$ on probability measures endowed with the Wasserstein distance. We recall some well-known facts about optimal transport that we will use below. For $\mu,\nu \in \mathcal P(I)$, we define $\Pi (\mu,\nu)$ as the set of the couplings of $\mu$ and $\nu$, that is, $\pi \in \mathcal P(I \times I)$ is such that
\[
\pi(\mathcal{B}\times I)=\mu(\mathcal{B}), \qquad  \pi(I \times \mathcal{B})=\nu(\mathcal{B})
\]
for each Borel set $\mathcal{B}$. Equivalently $\pi \in \Pi (\mu,\nu)$ if and only if 
 \begin{equation}\label{eq:special-property-W}
  \int _{I \times I}
( g (x) + \psi (y)) d\pi(x,y)=\int_{I}g (x)\mu(dx)+ \int_{I}\psi (y)\nu(dy), \quad \forall g,\psi \in C^0(I).
 \end{equation}

 The Wasserstein distance of order $1$ is defined by
\[
d_W(\mu,\nu)
:=
\inf_{\pi \in \Pi(\mu,\nu)}
\int_{I\times I} |x-y|\, d\pi(x,y).
\]
By the Kantorovich--Rubinstein duality theorem (see \cite[Theorem 5.10]{Villani2009}),
\begin{equation}\label{def:bL-dist}
d_W(\mu,\nu)
=
\sup_{\substack{\varphi:I\to\mathbb R\\ \operatorname{Lip}(\varphi)\le 1}}
\left|
\int_I \varphi\, d\mu
-
\int_I \varphi\, d\nu
\right|,
\end{equation}
 where \(
\operatorname{Lip}(\varphi)
:=
\sup_{\substack{x,y\in I\\ x\neq y}}
\frac{|\varphi(x)-\varphi(y)|}{|x-y|}.
\)

\begin{remark}\label{rmk:weak}
Since $I=[0,1]$ is compact, every probability measure on $I$ has finite first moment. 
In this setting, the Wasserstein distance $d_W$ induces the weak topology on $\mathcal P(I)$ 
(see \cite[Theorem 6.9]{Villani2009}). 
That is, for $\{\mu_n\}_{n\in\mathbb N}\subset \mathcal P(I)$,
\[
\mu_n \rightharpoonup \mu
\quad \Longleftrightarrow \quad
d_W(\mu_n,\mu)\to 0.
\]
\end{remark}

\subsection{Results}

We focus on invariant measures that are physically relevant, namely measures
which attract absolutely continuous initial distributions. We write
\(\mu\ll\nu\) to say that \(\mu\) is absolutely continuous with respect to \(\nu\),
and we denote by
\[
AC(I):=\{\mu\in \mathcal P(I):\mu\ll m\}
\]
the set of probability measures on \(I\) which are absolutely continuous with
respect to the Lebesgue measure \(m\).

\begin{definition}\label{def:physical}
An invariant measure \(\mu^\eps\) is called \emph{physical} if, for every
initial measure \(\mu_0\in AC(I)\), the sequence defined by
\[
\mu_{n+1}
=
(T\circ g_{\eps,\mu_n})_* \mu_n
\]
converges weakly to \(\mu^\eps\).
\end{definition}
The main result of this section is the following theorem, which establishes the existence and uniqueness of the physical invariant measure together with exponential convergence to equilibrium for the associated STO. 
\begin{theorem}\label{thm:main-physical}
For every \(\epsilon\in(0,1)\), the self-consistent transfer operator
\(\mathcal L\) admits a unique physical invariant measure
\(\mu^\epsilon\in AC(I)\). Moreover, for every \(\mu_0\in AC(I)\),
\[
d_W\bigl(\mathcal L^n\mu_0,\mu^\epsilon\bigr)
\le
\sigma_\epsilon^n d_W(\mu_0,\mu^\epsilon),
\qquad n\ge0,
\]
where
\[
\sigma_\epsilon:=\lambda^{-1}(1-\epsilon)\in(0,1).
\]
In particular,
\[
\mathcal L^n\mu_0\longrightarrow\mu^\epsilon
\]
exponentially fast in \(d_W\), and hence weakly, for every
\(\mu_0\in AC(I)\).
\end{theorem}
The proof of Theorem~\ref{thm:main-physical} is completed at the end of the section. The strategy is to first establish the contraction estimate and construct
the common attracting limit through the continuous extension of
\(\mathcal L|_{AC(I)}\), and then characterize the physical measure through the quantile function, so that Theorem \ref{thm:main-physical} will be a direct consequence of the following
contraction estimate.

\begin{proposition}\label{prop:equil}
Fix \(\epsilon\in(0,1)\). Then
\[
\mc L(AC(I))\subset AC(I),
\]
and for every \(\mu,\nu\in AC(I)\),
\[
d_W(\mc L\mu,\mc L\nu)
\le
\lambda^{-1}(1-\epsilon)d_W(\mu,\nu).
\]
In particular, \(\mc L\) is a strict contraction on \(AC(I)\), with contraction
rate
\[
\sigma_\epsilon:=\lambda^{-1}(1-\epsilon)\in(0,1).
\]
\end{proposition}

%%%%%%%%%%%%%%%

Assuming Proposition~\ref{prop:equil}, whose proof is postponed to
Subsection~\ref{sec:proof-prop-equil}, we prove the following lemma.

\begin{lemma}\label{lem:self}
For every \(\epsilon\in(0,1)\), the restriction
\(\mathcal L|_{AC(I)}\) admits a unique continuous extension
\[
\overline{\mathcal L}:\mathcal P(I)\to \mathcal P(I),
\]
and \(\overline{\mathcal L}\) has a unique fixed point
\(\overline\mu^\epsilon\in \mathcal P(I)\). Moreover, for every
\(\mu_0\in AC(I)\),
\[
d_W\bigl(\mathcal L^n\mu_0,\overline\mu^\epsilon\bigr)
\le
\sigma_\epsilon^n
d_W(\mu_0,\overline\mu^\epsilon),
\qquad n\ge0.
\]
In particular,
\[
\mathcal L^n\mu_0\longrightarrow\overline\mu^\epsilon
\]
in \(d_W\), and hence weakly, for every \(\mu_0\in AC(I)\).
\end{lemma}

\begin{proof}
Since \(AC(I)\) is dense in \(\mathcal P(I)\) with respect to \(d_W\), the
contraction
\[
\mathcal L|_{AC(I)}:AC(I)\to AC(I)
\]
admits a unique continuous extension
\[
\overline{\mathcal L}:\mathcal P(I)\to \mathcal P(I).
\]
The extension satisfies
\[
d_W(\overline{\mathcal L}\mu,\overline{\mathcal L}\nu)
\le
\sigma_\epsilon d_W(\mu,\nu),
\qquad \mu,\nu\in \mathcal P(I).
\]
Since \((\mathcal P(I),d_W)\) is complete, Banach's fixed-point theorem yields a
unique fixed point \(\overline\mu^\epsilon\in \mathcal P(I)\) of
\(\overline{\mathcal L}\), and
\[
d_W\bigl(\overline{\mathcal L}^{\,n}\mu_0,\overline\mu^\epsilon\bigr)
\le
\sigma_\epsilon^n d_W(\mu_0,\overline\mu^\epsilon).
\]
For \(\mu_0\in AC(I)\), Proposition~\ref{prop:equil} implies that
\(\mathcal L^n\mu_0\in AC(I)\) for every \(n\), and hence
\[
\overline{\mathcal L}^{\,n}\mu_0=\mathcal L^n\mu_0.
\]
The conclusion follows.
\end{proof}

The previous lemma identifies a unique probability measure attracting all
absolutely continuous initial distributions. However, it does not yet show
that this measure is invariant under the original self-consistent transfer
operator. Indeed, the previous lemma concerns the continuous extension
\(\overline{\mathcal L}\) of the restriction
\(\mathcal L|_{AC(I)}\), whereas the original STO on \(\mathcal P(I)\) may admit
additional atomic invariant measures (we exhibit an example later in Remark \ref{rmk:nonunique}). We therefore turn to a quantile
characterization of the attracting limit. This will show that the limit is
absolutely continuous and satisfies the fixed-point equation for the original
STO, thereby identifying it as the unique physical invariant measure.

%%%%%%%%%%%%%%%%

%%%%%%%%%%%%%%%%%%%%
%%%%quantile%%%%
\subsection{Quantile characterization of the physical measure}
For $\mu\in \mathcal P(I)$, we denote by
\begin{equation}\label{def:pseudo}
Q_\mu(u):=\inf\{x\in I:\ F_\mu(x)\ge u\},
\qquad u\in(0,1],
\end{equation}
the pseudo-inverse, or quantile function, of the cumulative distribution function
\[
F_\mu(x):=\mu([0,x]).
\]
At $u=0$, we set $Q_\mu(0):=\lim_{u\downarrow0}Q_\mu(u)$. With this convention, in dimension one we have
\begin{equation}\label{eq:d_w-dim1}
d_W(\mu,\nu)=\int_0^1 |Q_\mu(u)-Q_\nu(u)|\,du.
\end{equation}
Finally, note that when $\mu$ is nonatomic, $\mu([0,x))$ and $\mu([0,x])$ coincide.

Recalling \eqref{def:pseudo}, we prove the following easy fact on $Q_\mu$ which is the key to prove the next Lemma.\footnote{Although throughout this section $T$ is assumed to be a $C^1$
 strictly increasing embedding, the following lemma holds under the weaker assumptions of monotonicity and left continuity. We state it in this more general form since it may be of independent interest.}

\begin{lemma}\label{lem:quantile-push}
Let $T:I\to I$ be a non-decreasing left-continuous map. For every $\mu\in \mathcal P(I)$ and every $u\in(0,1)$,
\[
Q_{T_*\mu}(u)=T\big(Q_\mu(u)\big).
\]
\end{lemma}

\begin{proof}
Set $\nu:=T_*\mu$. For $y\in I$ we have, by definition of pushforward,
\[
F_\nu(y)=\nu([0,y])=(T_*\mu)([0,y])=\mu\big(T^{-1}([0,y])\big).
\]
If $y<T(0)$, the preimage is empty and $F_\nu(y)=0$. For $y\in[T(0),1]$, since $T$ is nondecreasing and left-continuous, the preimage $T^{-1}([0,y])$ is an interval of the form
$[0,\alpha(y)]$, where
\[
\alpha(y):=\sup\{x\in I:\ T(x)\le y\}.
\]
Hence
\begin{equation}\label{eq:CDF-push}
F_\nu(y)=\mu([0,\alpha(y)])=F_\mu(\alpha(y)).
\end{equation}
By definition,
\[
Q_\nu(u)=\inf\{y\in I:\ F_\nu(y)\ge u\}
=\inf\{y\in[T(0),1]:\ F_\mu(\alpha(y))\ge u\}.
\]
We claim that, for every $y\in[T(0),1]$,
\begin{equation}\label{eq:equiv-alpha}
F_\mu(\alpha(y))\ge u
\quad\Longleftrightarrow\quad
\alpha(y)\ge Q_\mu(u).
\end{equation}
Indeed, if $\alpha(y)<Q_\mu(u)$ then by minimality of $Q_\mu(u)$ we have
$F_\mu(\alpha(y))<u$; conversely if $\alpha(y)\ge Q_\mu(u)$, by monotonicity of $F_\mu$,
$F_\mu(\alpha(y))\ge F_\mu(Q_\mu(u))\ge u$.

Next, for $y\in[T(0),1]$, by the definition of $\alpha(y)$ and monotonicity of $T$ we have, for every $x\in I$,
\begin{equation}\label{eq:alpha-order}
\alpha(y)\ge x \quad\Longleftrightarrow\quad T(x)\le y.
\end{equation}
Applying \eqref{eq:alpha-order} with $x=Q_\mu(u)$ and combining with \eqref{eq:equiv-alpha}, we obtain
\[
F_\mu(\alpha(y))\ge u
\quad\Longleftrightarrow\quad
\alpha(y)\ge Q_\mu(u)
\quad\Longleftrightarrow\quad
T(Q_\mu(u))\le y.
\]
Therefore,
\[
Q_\nu(u)=\inf\{y\in[T(0),1]:\ T(Q_\mu(u))\le y\}=T(Q_\mu(u)),
\]
which proves the lemma.
\end{proof}

The following lemma characterizes the physical invariant measure through its
quantile function.
\begin{lemma}\label{lem:quantile-fixed-point}
There exists a unique nondecreasing measurable function \(Q^\epsilon:[0,1]\to I\)
satisfying
\[
Q^\epsilon(u)
=
T\big((1-\epsilon)Q^\epsilon(u)+\epsilon u\big),
\qquad u\in[0,1].
\]
Moreover, \(Q^\epsilon\) is strictly increasing and the probability measure \(\mu^\epsilon\) with \(Q^\epsilon\) as quantile function is absolutely continuous.
\end{lemma}

\begin{proof}
Let \(\mathcal Q\) be the set of continuous nondecreasing maps
\(Q:[0,1]\to I\), endowed with the supremum norm. This is a nonempty closed subset of \(C([0,1],\mathbb R)\), hence a complete metric space. Define
\[
(\mathcal FQ)(u)
:=
T\big((1-\epsilon)Q(u)+\epsilon u\big).
\]
Since \(T\) is continuous and increasing, \(\mathcal F\) maps \(\mathcal Q\) into itself.
Moreover, for any \(Q_1,Q_2\in\mathcal Q\),
\[
\|\mathcal FQ_1-\mathcal FQ_2\|_\infty
\le
\lambda^{-1}(1-\epsilon)\|Q_1-Q_2\|_\infty.
\]
Thus Banach's fixed-point theorem gives a unique fixed point
\(Q^\epsilon\in\mathcal Q\). If \(\widetilde Q:[0,1]\to I\) is any other measurable solution, then for every \(u\in[0,1]\),
\[
|\widetilde Q(u)-Q^\epsilon(u)|
\le \lambda^{-1}(1-\epsilon)|\widetilde Q(u)-Q^\epsilon(u)|,
\]
so \(\widetilde Q=Q^\epsilon\) pointwise.

Let \(\mu^\epsilon:=(Q^\epsilon)_*m\), where \(m\) is Lebesgue measure on \([0,1]\). We show that \(Q^\epsilon\) is its quantile function and \(\mu^\epsilon\in AC(I)\). Let
\[
\tau:=\inf_{x\in I}T'(x)>0.
\]
For \(0\le u<v\le1\), using the fixed-point equation and the monotonicity of
\(Q^\epsilon\), we obtain
\[
\begin{aligned}
Q^\epsilon(v)-Q^\epsilon(u)
&=
T\big((1-\epsilon)Q^\epsilon(v)+\epsilon v\big)
-
T\big((1-\epsilon)Q^\epsilon(u)+\epsilon u\big)\\
&\ge
\tau\Big((1-\epsilon)\big(Q^\epsilon(v)-Q^\epsilon(u)\big)
+\epsilon(v-u)\Big)\\
&\ge
\tau\epsilon(v-u).
\end{aligned}
\]
Hence \(Q^\epsilon\) is continuous and strictly increasing. The CDF of \(\mu^\epsilon\) equals \((Q^\epsilon)^{-1}\) on \([Q^\epsilon(0),Q^\epsilon(1)]\), extended by \(0\) below this interval and \(1\) above it. In particular, \(Q^\epsilon\) is the quantile function of \(\mu^\epsilon\) with our endpoint convention. The above lower Lipschitz bound implies
\[
F_{\mu^\epsilon}(y)-F_{\mu^\epsilon}(x)
\le
\frac{1}{\tau\epsilon}(y-x),
\qquad 0\le x<y\le1.
\]
Thus \(F_{\mu^\epsilon}\) is Lipschitz, and therefore
\(\mu^\epsilon\ll m\).
\end{proof}

%%%%
\begin{corollary}\label{cor:corunique}
The measure \(\mu^\epsilon\) constructed in
Lemma~\ref{lem:quantile-fixed-point} is invariant under \(\mathcal L\)
and is its unique nonatomic invariant probability measure. Moreover,
it coincides with the attracting measure \(\overline\mu^\epsilon\) of
Lemma~\ref{lem:self}.
\end{corollary}

\begin{proof}
We first prove invariance. Since \(\mu^\epsilon\) is nonatomic,
\[
F_{\mu^\epsilon}^-(x)=F_{\mu^\epsilon}(x),
\qquad x\in I,
\]
and
\[
F_{\mu^\epsilon}(Q^\epsilon(u))=u,
\qquad u\in(0,1).
\]
Therefore,
\[
g_{\epsilon,\mu^\epsilon}(Q^\epsilon(u))
=
(1-\epsilon)Q^\epsilon(u)+\epsilon u.
\]
By Lemma~\ref{lem:quantile-push},
\[
\begin{aligned}
Q_{\mathcal L\mu^\epsilon}(u)
&=
T\!\left(g_{\epsilon,\mu^\epsilon}(Q^\epsilon(u))\right)\\
&=
T\bigl((1-\epsilon)Q^\epsilon(u)+\epsilon u\bigr)\\
&=
Q^\epsilon(u).
\end{aligned}
\]
Since a probability measure is uniquely determined by its quantile function,
\[
\mathcal L\mu^\epsilon=\mu^\epsilon.
\]

We next prove uniqueness in the nonatomic class. Let
\(\nu\in\mathcal P(I)\) be nonatomic and suppose that
\(\mathcal L\nu=\nu\). As above, for every \(u\in(0,1)\),
\[
F_\nu^-(Q_\nu(u))=F_\nu(Q_\nu(u))=u.
\]
Hence
\[
\begin{aligned}
Q_\nu(u)
&=Q_{\mathcal L\nu}(u)\\
&=T\bigl((1-\epsilon)Q_\nu(u)+\epsilon u\bigr).
\end{aligned}
\]
For each fixed \(u\), the map
\[
x\longmapsto T\bigl((1-\epsilon)x+\epsilon u\bigr)
\]
is a contraction with Lipschitz constant at most
\(\lambda^{-1}(1-\epsilon)<1\). Its fixed point is therefore unique.
By Lemma~\ref{lem:quantile-fixed-point},
\[
Q_\nu(u)=Q^\epsilon(u),
\qquad u\in(0,1),
\]
and consequently \(\nu=\mu^\epsilon\).

Finally, since \(\mu^\epsilon\in AC(I)\), the extension
\(\overline{\mathcal L}\) agrees with \(\mathcal L\) at
\(\mu^\epsilon\). Thus
\[
\overline{\mathcal L}\mu^\epsilon
=
\mathcal L\mu^\epsilon
=
\mu^\epsilon.
\]
By uniqueness of the fixed point of \(\overline{\mathcal L}\),
established in Lemma~\ref{lem:self}, we conclude that
\[
\mu^\epsilon=\overline\mu^\epsilon.
\]
\end{proof}

\begin{remark}\label{rmk:nonunique}
The uniqueness statement of Corollary~\ref{cor:corunique} does not
extend to the whole space \(\mathcal P(I)\). Indeed, consider the map
\[
H_\epsilon(a):=T\bigl((1-\epsilon)a\bigr),
\qquad a\in I.
\]
It is a contraction of \(I\), with Lipschitz constant at most
\[
\lambda^{-1}(1-\epsilon)<1.
\]
It therefore has a unique fixed point \(a_\epsilon\in I\), satisfying
\[
a_\epsilon=T\bigl((1-\epsilon)a_\epsilon\bigr).
\]
For \(\mu=\delta_{a_\epsilon}\), our convention
\[
F_\mu^-(x)=\mu([0,x))
\]
gives
\[
F_\mu^-(a_\epsilon)=0
\]
and hence
\[
g_{\epsilon,\mu}(a_\epsilon)
=
(1-\epsilon)a_\epsilon.
\]
Consequently,
\[
\mathcal L\delta_{a_\epsilon}
=
\delta_{T((1-\epsilon)a_\epsilon)}
=
\delta_{a_\epsilon}.
\]
Thus the original self-consistent transfer operator always admits an
atomic invariant measure in addition to the unique nonatomic invariant
measure \(\mu^\epsilon\). The latter is distinguished by its physical
attraction property.
\end{remark}
We are left with the proof of Proposition \ref{prop:equil}.
\subsection{Proof of Proposition \ref{prop:equil}}\label{sec:proof-prop-equil}
The proof of Proposition~\ref{prop:equil} relies on the following lemma.

%%%%%%%

\begin{lemma}\label{lem:coupling-contraction}
Let \(\mu,\nu\in \mathcal P(I)\) be nonatomic probability measures; fix $\epsilon\in(0,1)$. 
Denote by $m$ the Lebesgue measure on $I$. Then the following identities hold:
\begin{align}
d_{W}\big(m,L_{g_{\epsilon,\mu}}\mu\big)
&= (1-\epsilon)\,d_{W}(m,\mu),\label{meas3-quant}\\
d_{W}\big(L_{g_{\epsilon,\nu}}\nu,\;L_{g_{\epsilon,\mu}}\mu\big)
&= (1-\epsilon)\,d_{W}(\nu,\mu).\label{meas4-quant}
\end{align}
\end{lemma}

\begin{proof}
We use the pseudo-inverse
representation \eqref{def:pseudo}.
For any fixed nonatomic probability measure \(\mu\), the map
\[
g_{\epsilon,\mu}(x)
=
(1-\epsilon)x+\epsilon F_\mu^-(x)
\]
is nondecreasing and continuous. Since \(\mu\) is nonatomic,
\[
F_\mu^-(x)=F_\mu(x),
\qquad x\in I,
\]
and therefore
\[
g_{\epsilon,\mu}(Q_\mu(u))
=
(1-\epsilon)Q_\mu(u)+\epsilon u.
\] Hence
\begin{equation}\label{eq:quantile-coupling}
Q_{(g_{\epsilon,\mu})_*\mu}(u)
= g_{\epsilon,\mu}\big(Q_\mu(u)\big)
= (1-\epsilon)Q_\mu(u)+\epsilon u,
\qquad u\in[0,1].
\end{equation}

The first equality follows from Lemma~\ref{lem:quantile-push}. For the second equality note that, by nonatomicity, for every \(u\in(0,1)\),
\[
F_\mu(Q_\mu(u))=u.
\]

Since \(Q_m(u)=u\), identity \eqref{eq:quantile-coupling} yields
\[
Q_{L_{g_{\epsilon,\mu}}\mu}(u)=(1-\epsilon)Q_\mu(u)+\epsilon u.
\]
Therefore, by \eqref{eq:d_w-dim1},
\[
\begin{aligned}
d_W\big(m,L_{g_{\epsilon,\mu}}\mu\big)
&=\int_0^1 \big|Q_m(u)-Q_{L_{g_{\epsilon,\mu}}\mu}(u)\big|\,du\\
&=\int_0^1 \big|u-[(1-\epsilon)Q_\mu(u)+\epsilon u]\big|\,du\\
&=\int_0^1 (1-\epsilon)\big|u-Q_\mu(u)\big|\,du\\
&=(1-\epsilon)\,d_W(m,\mu),
\end{aligned}
\]
which proves \eqref{meas3-quant}.

Similarly, applying the same formula to both $\mu$ and $\nu$ gives
\[
Q_{L_{g_{\epsilon,\nu}}\nu}(u)=(1-\epsilon)Q_\nu(u)+\epsilon u,
\qquad
Q_{L_{g_{\epsilon,\mu}}\mu}(u)=(1-\epsilon)Q_\mu(u)+\epsilon u.
\]
Therefore,
\[
\begin{aligned}
d_W\big(L_{g_{\epsilon,\nu}}\nu,\;L_{g_{\epsilon,\mu}}\mu\big)
&=\int_0^1 \big|(1-\epsilon)\big(Q_\nu(u)-Q_\mu(u)\big)\big|\,du\\
&=(1-\epsilon)\int_0^1 |Q_\nu(u)-Q_\mu(u)|\,du\\
&=(1-\epsilon)\,d_W(\nu,\mu),
\end{aligned}
\]
which proves \eqref{meas4-quant} and completes the proof.
\end{proof}
We are now ready to prove Proposition \ref{prop:equil}, which at this point is just the standard proof of the Lipschitz
stability of Wasserstein distance under push-forwards, and it is presented for completeness.

\begin{proof}
For \(\mu\in AC(I)\), the map \(g_{\epsilon,\mu}\) is continuous, maps \(I\) onto itself, and satisfies
\[
g_{\epsilon,\mu}(y)-g_{\epsilon,\mu}(x)
\ge (1-\epsilon)(y-x),\qquad x<y.
\]
Thus its inverse is Lipschitz. Since \(\inf_I T'>0\), the inverse of \(T\) on \(T(I)\) is also Lipschitz. Consequently, the inverse of \(T\circ g_{\epsilon,\mu}\) on \(T(I)\) is Lipschitz and maps Lebesgue-null sets to Lebesgue-null sets. Since \(\mu\ll m\), it follows that
\[
\mathcal L(\mu)
=
(T\circ g_{\epsilon,\mu})_*\mu
\]
is absolutely continuous with respect to the Lebesgue measure, that is,
\(
\mathcal L(\mu)\in AC(I).
\)
\\
Next, let $\mu,\nu\in AC(I)$ and fix $\epsilon\in(0,1)$. 
Recall that
\[
\mc L(\mu)
=
L_T\big(L_{g_{\epsilon,\mu}}(\mu)\big)
=
(T\circ g_{\epsilon,\mu})_*
\mu.
\]

By Lemma~\ref{lem:coupling-contraction} we have
\[
d_W\big(L_{g_{\epsilon,\nu}}\nu,\;L_{g_{\epsilon,\mu}}\mu\big)
\le
(1-\epsilon)\,d_W(\nu,\mu).
\]

It remains to estimate the effect of $L_T$. 
Since $T$ satisfies $0<T'(x)\le \lambda^{-1}<1$, it is $\lambda^{-1}$-Lipschitz, so
\[
\operatorname{Lip}(\varphi\circ T)
\le
\lambda^{-1}\operatorname{Lip}(\varphi).
\]
Hence, if \(\operatorname{Lip}(\varphi)\le1\),
\[
\operatorname{Lip}(\varphi\circ T)\le\lambda^{-1}.
\]
Using \eqref{def:bL-dist},
we obtain for any $\mu_1,\mu_2\in \mathcal P(I)$
\[
\begin{aligned}
d_W(L_T\mu_1,L_T\mu_2)
&=
\sup_{\operatorname{Lip}(\varphi)\le1}
\left|
\int \varphi\, d(L_T\mu_1)
-
\int \varphi\, d(L_T\mu_2)
\right|\\
&=
\sup_{\operatorname{Lip}(\varphi)\le1}
\left|
\int \varphi\circ T\, d\mu_1
-
\int \varphi\circ T\, d\mu_2
\right|\\
&\le
\lambda^{-1}\, d_W(\mu_1,\mu_2).
\end{aligned}
\]

Applying this estimate with 
$\mu_1=L_{g_{\epsilon,\nu}}\nu$ and 
$\mu_2=L_{g_{\epsilon,\mu}}\mu$, we conclude
\[
d_W(\mc L\nu,\mc L\mu)
\le
\lambda^{-1}\,
d_W\big(L_{g_{\epsilon,\nu}}\nu,\;L_{g_{\epsilon,\mu}}\mu\big)
\le
\lambda^{-1}(1-\epsilon)\,d_W(\nu,\mu).
\]

This proves Proposition~\ref{prop:equil}.
\end{proof}

\begin{proof}[Proof of Theorem~\ref{thm:main-physical}]
By Corollary~\ref{cor:corunique}, the measure
\(\mu^\epsilon\in AC(I)\) is invariant under the original
self-consistent transfer operator and coincides
with the attracting measure \(\overline\mu^\epsilon\) of
Lemma~\ref{lem:self}. Therefore, for every \(\mu_0\in AC(I)\),
\[
d_W\bigl(\mathcal L^n\mu_0,\mu^\epsilon\bigr)
\le
\sigma_\epsilon^n d_W(\mu_0,\mu^\epsilon).
\]
In particular, \(\mu^\epsilon\) is physical.

Suppose that \(\nu^\epsilon\) is another physical invariant measure.
For any \(\mu_0\in AC(I)\), the definition of physicality and the
estimate above give
\[
\mathcal L^n\mu_0\rightharpoonup\nu^\epsilon
\qquad\text{and}\qquad
\mathcal L^n\mu_0\rightharpoonup\mu^\epsilon.
\]
Since weak limits are unique,
\[
\nu^\epsilon=\mu^\epsilon.
\]
\end{proof}

%%%%%%%
\section{Finite-particle system and thermodynamic limit}\label{sec:finiteN}

\subsection{Approximation by the finite-particle equilibrium}
This section concerns the long-time behavior of the finite system itself. In particular, one may ask whether the equilibrium configurations selected by the finite-$N$ dynamics approximate the invariant measure of the continuum STO. In this section we show that the continuum equilibrium naturally induces a discrete equilibrium configuration for the finite particle system. Moreover, we prove that the associated empirical measure converges to the invariant measure of the STO with an explicit rate of order $1/N$ in Wasserstein distance.

For $N\ge1$, the finite-particle dynamics is defined by
\[
x_i^{(n+1)}
=T\!\left(
(1-\epsilon)x_i^{(n)}
+\frac{\epsilon}{N}
\#\{0\le j\le N-1:\ x_j^{(n)}<x_i^{(n)}\}
\right),
\qquad 0\le i\le N-1.
\]
The strict inequality agrees with the left-continuous CDF convention.

The following estimate shows that the continuum equilibrium is recovered by sampling its quantile at equally spaced ranks.
\begin{proposition}\label{prop:finiteN-equilibrium}
Let $Q^\epsilon$ be the quantile fixed point from Lemma~\ref{lem:quantile-fixed-point}, and define
\[
x_i^{\epsilon,N}:=Q^\epsilon\!\left(\frac{i}{N}\right),
\qquad i=0,\dots,N-1,
\]
and set
\[
\mu^{\epsilon,N}:=\frac1N\sum_{i=0}^{N-1}\delta_{x_i^{\epsilon,N}}.
\]
Then the configuration
\(
(x_0^{\epsilon,N},\dots,x_{N-1}^{\epsilon,N})
\)
is a fixed point of the finite-$N$ dynamics. Equivalently, the associated
empirical measure satisfies
\[
\mathcal L(\mu^{\epsilon,N})=\mu^{\epsilon,N}.
\]
\end{proposition}

\begin{proof}
Since $Q^\epsilon$ is strictly increasing, the points
$x_i^{\epsilon,N}=Q^\epsilon(i/N)$ are ordered increasingly and distinct. Hence
\[
\mu^{\epsilon,N}([0,x_i^{\epsilon,N}))=\frac{i}{N}.
\]
Therefore,
\[
g_{\epsilon,\mu^{\epsilon,N}}(x_i^{\epsilon,N})
=
(1-\epsilon)x_i^{\epsilon,N}
+
\epsilon\frac{i}{N}.
\]
Using the fixed point equation for $Q^\epsilon$ at $u=i/N$, we get
\[
T\!\left(g_{\epsilon,\mu^{\epsilon,N}}(x_i^{\epsilon,N})\right)
=
T\!\left((1-\epsilon)Q^\epsilon\!\left(\frac{i}{N}\right)
+\epsilon\frac{i}{N}\right)
=
Q^\epsilon\!\left(\frac{i}{N}\right)
=
x_i^{\epsilon,N}.
\]
Thus every atom of $\mu^{\epsilon,N}$ is fixed by the map
$T\circ g_{\epsilon,\mu^{\epsilon,N}}$. Consequently,
\[
\mc L(\mu^{\epsilon,N})
=
(T\circ g_{\epsilon,\mu^{\epsilon,N}})_*\mu^{\epsilon,N}
=
\mu^{\epsilon,N}.
\]
\end{proof}
The previous proposition shows that the quantile discretization is not merely a numerical approximation of the physical equilibrium, but in fact gives an equilibrium of the finite-particle system with distinct ordered particles.

\begin{proposition}\label{prop:finiteN-attraction}
Let
\[
x_0^{(0)}<x_1^{(0)}<\cdots<x_{N-1}^{(0)}
\]
and let \(\{x_i^{(n)}\}\) evolve according to the finite-particle
dynamics. Then
\[
x_0^{(n)}<x_1^{(n)}<\cdots<x_{N-1}^{(n)}
\qquad\text{for every }n\ge0,
\]
and
\[
\max_{0\le i\le N-1}
\left|x_i^{(n)}-x_i^{\epsilon,N}\right|
\le
\sigma_\epsilon^n
\max_{0\le i\le N-1}
\left|x_i^{(0)}-x_i^{\epsilon,N}\right|,
\]
where
\[
\sigma_\epsilon=\lambda^{-1}(1-\epsilon)<1.
\]
In particular, the configuration
\[
(x_0^{\epsilon,N},\ldots,x_{N-1}^{\epsilon,N})
\]
is the unique equilibrium among strictly ordered configurations and is
exponentially attracting.

If
\[
\eta_N^{(n)}
=
\frac1N\sum_{i=0}^{N-1}\delta_{x_i^{(n)}},
\]
then
\[
d_W(\eta_N^{(n)},\mu^{\epsilon,N})
\le
\sigma_\epsilon^n
d_W(\eta_N^{(0)},\mu^{\epsilon,N}).
\]
\end{proposition}

\begin{proof}
The ordering is preserved because
\[
x_i^{(n+1)}
=
T\left((1-\epsilon)x_i^{(n)}+\epsilon\frac{i}{N}\right)
\]
and both the argument of \(T\) and \(T\) itself are strictly increasing
with respect to \(i\).

Since \(x_i^{\epsilon,N}\) is an equilibrium,
\[
x_i^{\epsilon,N}
=
T\left((1-\epsilon)x_i^{\epsilon,N}
+\epsilon\frac{i}{N}\right).
\]
Therefore,
\[
\begin{aligned}
\left|x_i^{(n+1)}-x_i^{\epsilon,N}\right|
&\le
\lambda^{-1}(1-\epsilon)
\left|x_i^{(n)}-x_i^{\epsilon,N}\right|\\
&=
\sigma_\epsilon
\left|x_i^{(n)}-x_i^{\epsilon,N}\right|.
\end{aligned}
\]
Iteration proves the first estimate and also uniqueness of the
equilibrium in the strictly ordered sector.

Since both empirical measures have increasingly ordered atoms of equal
weight,
\[
d_W(\eta_N^{(n)},\mu^{\epsilon,N})
=
\frac1N\sum_{i=0}^{N-1}
\left|x_i^{(n)}-x_i^{\epsilon,N}\right|.
\]
Applying the same coordinatewise estimate proves the final assertion.
\end{proof}

We now quantify the approximation error between the two equilibria.

\begin{theorem}\label{thm:fixedpoint-approx}
Let \(\mu^\epsilon\) be the invariant measure in
Corollary~\ref{cor:corunique}, and let
\[
\tau:=\min_{x\in I}T'(x)>0.
\]
Then
\[
\frac{\tau\epsilon}{2N}
\le
d_W(\mu^{\epsilon,N},\mu^\epsilon)
\le
\frac{M_\epsilon}{2N},
\]
where
\[
M_\epsilon
:=
\frac{\lambda^{-1}\epsilon}
{1-\lambda^{-1}(1-\epsilon)}.
\]
In particular, the finite-particle equilibrium converges to the
nonatomic equilibrium in \(d_W\) with optimal order \(N^{-1}\).
\end{theorem}

\begin{proof}
For \(0\le u<v\le1\), the fixed-point equation for \(Q^\epsilon\)
gives
\[
\begin{aligned}
Q^\epsilon(v)-Q^\epsilon(u)
&=
T\bigl((1-\epsilon)Q^\epsilon(v)+\epsilon v\bigr)\\
&\quad-
T\bigl((1-\epsilon)Q^\epsilon(u)+\epsilon u\bigr).
\end{aligned}
\]
Using \(T'\le\lambda^{-1}\), we obtain
\[
Q^\epsilon(v)-Q^\epsilon(u)
\le
\lambda^{-1}
\Bigl(
(1-\epsilon)\bigl(Q^\epsilon(v)-Q^\epsilon(u)\bigr)
+\epsilon(v-u)
\Bigr).
\]
Therefore,
\[
Q^\epsilon(v)-Q^\epsilon(u)
\le
M_\epsilon(v-u).
\]
On the other hand, since \(T'\ge\tau\) and \(Q^\epsilon\) is
nondecreasing,
\[
Q^\epsilon(v)-Q^\epsilon(u)
\ge
\tau\epsilon(v-u).
\]

For almost every
\[
u\in\left(\frac{i}{N},\frac{i+1}{N}\right],
\]
the quantile function of \(\mu^{\epsilon,N}\) is
\[
Q_{\mu^{\epsilon,N}}(u)
=
Q^\epsilon\left(\frac{i}{N}\right).
\]
Consequently,
\[
\begin{aligned}
d_W(\mu^{\epsilon,N},\mu^\epsilon)
&=
\sum_{i=0}^{N-1}
\int_{i/N}^{(i+1)/N}
\left[
Q^\epsilon(u)
-
Q^\epsilon\left(\frac{i}{N}\right)
\right],du.
\end{aligned}
\]
For \(u\in(i/N,(i+1)/N]\), the preceding two-sided Lipschitz estimate
gives
\[
\tau\epsilon\left(u-\frac{i}{N}\right)
\le
Q^\epsilon(u)-Q^\epsilon\left(\frac{i}{N}\right)
\le
M_\epsilon\left(u-\frac{i}{N}\right).
\]
Integrating and summing over \(i\) yields
\[
\frac{\tau\epsilon}{2N}
\le
d_W(\mu^{\epsilon,N},\mu^\epsilon)
\le
\frac{M_\epsilon}{2N}.
\]
\end{proof}

%%%%%%%THMD LMT%%%%%%
\subsection{Thermodynamic limit}
We now turn to the evolution itself and investigate the relation between the finite particle dynamics and the self-consistent transfer operator uniformly over all time steps.

Since empirical measures are atomic whereas the contraction result of Proposition~\ref{prop:equil} applies to absolutely continuous measures, a direct comparison is not available. The following lemma provides the required mixed estimate, comparing one step of the finite particle dynamics with one step of the continuum evolution.
\begin{lemma}\label{lem:mixed-finite-continuum}
Let $\epsilon \in (0,1)$,  $\mu\in AC(I)$ and let
\[
\eta_N=\frac1N\sum_{i=0}^{N-1}\delta_{x_{(i)}},
\qquad
x_{(0)}<x_{(1)}<\cdots<x_{(N-1)},
\]
be an empirical measure with distinct ordered atoms. Then
\[
d_W\big(L_{g_{\epsilon,\eta_N}}\eta_N,\,
L_{g_{\epsilon,\mu}}\mu\big)
\le
(1-\epsilon)d_W(\eta_N,\mu)+\frac{\epsilon}{2N}.
\]
Consequently,
\[
d_W\big(\mc L(\eta_N),\mc L(\mu)\big)
\le
\lambda^{-1}(1-\epsilon)d_W(\eta_N,\mu)
+
\frac{\lambda^{-1}\epsilon}{2N}.
\]
\end{lemma}

\begin{proof}
Let $Q_\mu$ and $Q_{\eta_N}$ denote the quantile functions. Since $\eta_N$ has ordered distinct atoms,
\[
Q_{\eta_N}(u)=x_{(i)}
\qquad\text{for }u\in\Big(\frac{i}{N},\frac{i+1}{N}\Big],
\]
and
\[
g_{\epsilon,\eta_N}(x_{(i)})
=
(1-\epsilon)x_{(i)}+\epsilon\frac{i}{N}.
\]
Hence
\[
Q_{L_{g_{\epsilon,\eta_N}}\eta_N}(u)
=
(1-\epsilon)Q_{\eta_N}(u)+\epsilon\frac{i}{N},
\qquad
u\in\Big(\frac{i}{N},\frac{i+1}{N}\Big].
\]
On the other hand, since $\mu$ is nonatomic,
\[
Q_{L_{g_{\epsilon,\mu}}\mu}(u)
=
(1-\epsilon)Q_\mu(u)+\epsilon u.
\]

Therefore,
\[
\begin{aligned}
&d_W\big(L_{g_{\epsilon,\eta_N}}\eta_N,\,
L_{g_{\epsilon,\mu}}\mu\big)\\
&\quad=
\sum_{i=0}^{N-1}
\int_{i/N}^{(i+1)/N}
\left|
(1-\epsilon)\big(Q_{\eta_N}(u)-Q_\mu(u)\big)
+
\epsilon\Big(\frac{i}{N}-u\Big)
\right|\,du\\
&\le
(1-\epsilon)\int_0^1 |Q_{\eta_N}(u)-Q_\mu(u)|\,du
+
\epsilon\sum_{i=0}^{N-1}\int_{i/N}^{(i+1)/N}
\left|u-\frac{i}{N}\right|\,du\\
&=
(1-\epsilon)d_W(\eta_N,\mu)
+
\frac{\epsilon}{2N}.
\end{aligned}
\]
Finally, since $T$ is $\lambda^{-1}$-Lipschitz,
\[
d_W(L_T\alpha,L_T\beta)\le \lambda^{-1}d_W(\alpha,\beta)
\]
for all $\alpha,\beta\in \mathcal P(I)$. Applying this with
\(\alpha=L_{g_{\epsilon,\eta_N}}\eta_N\) and
\(\beta=L_{g_{\epsilon,\mu}}\mu\) gives the second estimate.
\end{proof}
Finally, iterating the one-step estimate of Lemma~\ref{lem:mixed-finite-continuum} allows us to compare the entire finite particle evolution with the continuum dynamics. This yields the following quantitative thermodynamic limit.
\begin{theorem}\label{thm:thermodynamic-limit}
Let $\epsilon \in (0,1)$, $\mu_0\in AC(I)$, and let
\[
\mu^{(n+1)}=\mc L(\mu^{(n)}),\qquad \mu^{(0)}=\mu_0.
\]
Let $\eta_N^{(0)}$ be an empirical measure with $N$ distinct atoms such that
$d_W(\eta_N^{(0)},\mu_0)\to0$, and define
\[
\eta_N^{(n+1)}=\mc L(\eta_N^{(n)}).
\]
Then, for every $n\ge0$,
\[
d_W(\eta_N^{(n)},\mu^{(n)})
\le
\sigma_\epsilon^n d_W(\eta_N^{(0)},\mu_0)
+
\frac{\lambda^{-1}\epsilon}{2N}
\frac{1-\sigma_\epsilon^n}{1-\sigma_\epsilon},
\]
where
$
\sigma_\epsilon:=\lambda^{-1}(1-\epsilon)<1.
$
In particular, for every fixed $n\ge0$,
\[
\eta_N^{(n)}\longrightarrow \mu^{(n)}
\qquad\text{in }d_W.
\]
Moreover,
\[
\sup_{n\ge0}d_W(\eta_N^{(n)},\mu^{(n)})
\le
d_W(\eta_N^{(0)},\mu_0)
+
\frac{\lambda^{-1}\epsilon}{2N(1-\sigma_\epsilon)}.
\]
Consequently,
\[
\sup_{n\ge0}d_W(\eta_N^{(n)},\mu^{(n)})\longrightarrow0
\qquad\text{as }N\to\infty.
\]
\end{theorem}

\begin{proof}
By induction, each \(\eta_N^{(n)}\) has \(N\) distinct ordered atoms. Indeed,
if
\[
\eta_N^{(n)}=\frac1N\sum_{i=0}^{N-1}\delta_{x_i^{(n)}},
\qquad
x_0^{(n)}<\cdots<x_{N-1}^{(n)},
\]
then
\[
x_i^{(n+1)}
=
T\left((1-\epsilon)x_i^{(n)}+\epsilon\frac{i}{N}\right),
\]
and the strict monotonicity of \(T\) implies
\[
x_0^{(n+1)}<\cdots<x_{N-1}^{(n+1)}.
\]
Moreover, by Proposition~\ref{prop:equil}, \(\mu^{(n)}\in AC(I)\) for every
\(n\ge0\).
Set
\[
e_n:=d_W(\eta_N^{(n)},\mu^{(n)}).
\]
By Lemma~\ref{lem:mixed-finite-continuum},
\[
e_{n+1}
\le
\sigma_\epsilon e_n
+
\frac{\lambda^{-1}\epsilon}{2N}.
\]
Iterating this inequality gives
\[
e_n
\le
\sigma_\epsilon^n e_0
+
\frac{\lambda^{-1}\epsilon}{2N}
\sum_{j=0}^{n-1}\sigma_\epsilon^j
=
\sigma_\epsilon^n d_W(\eta_N^{(0)},\mu_0)
+
\frac{\lambda^{-1}\epsilon}{2N}
\frac{1-\sigma_\epsilon^n}{1-\sigma_\epsilon}.
\]
The remaining conclusions follow immediately.
\end{proof}

\begin{remark}
Combining Theorems~\ref{thm:thermodynamic-limit} and
\ref{thm:fixedpoint-approx}, we conclude that the finite-particle equilibrium
\(\mu^{\epsilon,N}\) with distinct ordered particles converges to the physical
equilibrium of the continuum model as \(N\to\infty\). Thus the transient dynamics
from the distinct-particle initial configurations considered above and these
equilibria are described by the self-consistent transfer operator in the
thermodynamic limit.
\end{remark}

\noindent \textbf{Acknowledgements}

 S.G. acknowledges the MIUR Excellence Department Project awarded to the
Department of Mathematics, University of Pisa, CUP I57G22000700001. S.G. was
partially supported by the research project "Stochastic properties of
dynamical systems" (PRIN 2022NTKXCX) of the Italian Ministry of Education
and Research. M.T. acknowledges Marie Slodowska-Curie Actions: "Ergodic
Theory of Complex Systems", project no. 843880. The authors acknowledge the UMI Group “DinAmicI” ({www.dinamici.org}) and the INdAM group GNFM.

\end{document}